\documentclass[12pt,A4,reqno]{amsart}
\usepackage{amsfonts}
\usepackage{mathrsfs}
\usepackage[T1]{fontenc}
\usepackage{amssymb}
\usepackage{pb-diagram}
\usepackage{amsmath,amscd}
\usepackage{upgreek}
\usepackage{mathabx}
\usepackage{color}
\usepackage{epsf}
\usepackage{graphicx}
\usepackage{xypic}

\theoremstyle{plain}
\newtheorem{thm}{Theorem}[section]
\newtheorem{lem}[thm]{Lemma}

\theoremstyle{definition}
\newtheorem{defi}[thm]{Definition}
\newtheorem{rem}[thm]{Remark}
\newtheorem{exam}{Example}

\newcommand{\R}{\mathbb R}
\newcommand{\Z}{\mathbb Z}

\newcommand{\mS}{\mathcal S}
\newcommand{\K}{\mathcal K}
\newcommand{\I}{\mathrm I}
\newcommand{\J}{\mathrm J}
\newcommand{\nn}{\vskip 0.2cm}
\newcommand{\n}{\vskip 0.1cm}

\begin{document}

\title [\ ] {On Hochster's formula for
         a class of quotient spaces of moment-angle complexes}

\author{Li Yu}
\address{Department of Mathematics and IMS, Nanjing University, Nanjing, 210093, P.R.China
 }
 \email{yuli@nju.edu.cn}
 \thanks{2010 \textit{Mathematics Subject Classification}. 57S15, 14M25, \\
 The author is partially supported by
 Natural Science Foundation of China (grant no.11371188) and the
 PAPD (priority academic program development) of Jiangsu higher education institutions.}


\keywords{Moment-angle complex, Hochster's formula, Tor algebra, balanced simple polytope}


\begin{abstract}
  Any finite simplicial complex $\mathcal{K}$ and a partition 
 of the vertex set of $\mathcal{K}$ determines a canonical quotient space 
 of the moment-angle complex of $\K$.
  We prove that the cohomology groups of such a space can be computed 
  via some Hochster's type formula, which generalizes the usual Hochster's
   formula for the cohomology groups of moment-angle complexes. In addition, we show that
  the stable decomposition of moment-angle complexes can also be 
  extended to such spaces. This type of spaces include 
  all the quasitoric manifolds that are pullback from the linear models.
 And we prove that
   the moment-angle complex associated to a finite simplicial poset is always homotopy
   equivalent to one of such spaces.
 \end{abstract}

\maketitle

 \section{Introduction} \label{Sec:Penal-Struct}
 
  An \emph{abstract simplicial complex} on a set $[m]=\{ v_1,\cdots, v_m\}$ 
  is a collection $\mathcal{K}$ of subsets $\sigma\subseteq [m]$ such that
  if $\sigma\in\mathcal{K}$, then any subset of $\sigma$ also belongs to $\mathcal{K}$.
  We always assume that the empty set belongs to $\mathcal{K}$ and refer
  to $\sigma\in \mathcal{K}$ as an abstract \emph{simplex} of $\mathcal{K}$.
  The simplex corresponding to the empty set is denoted by $\hat{\mathbf{0}}$.
  In particular, any element of $[m]$ is 
  called a \emph{vertex} of $\mathcal{K}$.  
  We call the number of
  vertices of a simplex $\sigma$ the \emph{rank} of $\sigma$, denoted by $\mathrm{rank}(\sigma)$.
  Let $\dim(\sigma)$ denote the dimension of a simplex $\sigma$.
   So $\mathrm{rank}(\sigma)=\dim(\sigma)+1$.\n
   
    Any finite abstract simplicial complex $\K$ admits a
  \emph{geometric realization} in some Euclidean space. 
  But sometimes we also use $\K$ to denote its geometric realization 
  when the meaning is clear in the context.\n

  Given a finite abstract simplicial complex $\mathcal{K}$ on a set $[m]$ and 
  a pair of spaces $(X,A)$ with $A\subset X$, we can construct of a topological space 
  $(X,A)^{\mathcal{K}}$ by:
    \begin{equation}\label{Equ:Construction}
     (X,A)^{\mathcal{K}}= \bigcup_{\sigma\in \mathcal{K}} (X,A)^{\sigma}, \ \text{where}\
      (X,A)^{\sigma}= \prod_{v_j\in \sigma} X \times \prod_{v_j\notin \sigma} A. 
    \end{equation}  
  The symbol $\prod$ here and in the rest of this paper means Cartesian product.  
  So $(X,A)^{\mathcal{K}}$ is a subspace of the Cartesian product of
  $m$ copies of $X$. It is called the \emph{polyhedral product}
   or the \emph{generalized moment-angle complex}
   of $\mathcal{K}$ and $(X,A)$. 
  In particular, $\mathcal{Z}_{\mathcal{K}}
    =(D^2,S^1)^{\mathcal{K}}$ and $\R\mathcal{Z}_{\mathcal{K}}
    =(D^1,S^0)^{\mathcal{K}}$ are called the 
  \emph{moment-angle complex} and \emph{real moment-angle complex} of $\mathcal{K}$,
   respectively (see~\cite{BP02}).
  Moreover, we can define the polyhedral product 
  $(\underline{\mathbb{X}},\underline{\mathbb{A}})^{\K}$ of $\K$ with $m$ pairs of spaces
  $(\underline{\mathbb{X}},\underline{\mathbb{A}}) = \{ (X_1,A_1),\cdots, (X_m,A_m) \}$ 
  (see~\cite{BBCG10} or~\cite[Sec 4.2]{BP15}).\n

    Originally, $\mathcal{Z}_{\mathcal{K}}$ and $\R\mathcal{Z}_{\mathcal{K}}$ were 
    constructed by Davis and Januszkiewicz~\cite{DaJan91} 
    in a different way. We will only explain the construction of $\mathcal{Z}_{\mathcal{K}}$
    below (the $\R\mathcal{Z}_{\mathcal{K}}$ case is completely parallel).    
   Let $\K'$ denote the barycentric subdivision of $\K$. We can consider
    $\K'$
   as the set of chains of simplices in $\K$ ordered by inclusions. 
     For each simplex
   $\sigma\in \K$, let $F_{\sigma}$ denote the geometric realization of the poset
   $\K_{\geq \sigma}=\{ \tau\in \K\,|\, \sigma\subseteq \tau \}$. Thus,  
    $F_{\sigma}$
   is the subcomplex of $\K'$ consisting of all simplices of the form
   $\sigma=\sigma_0 \subsetneq \sigma_1 \subsetneq \cdots \subsetneq \sigma_l$.    
    Let $P_{\K}$ denote the cone on $\K'$.   
    If $\sigma$ is a $(k-1)$-simplex, then we say that $F_{\sigma}\subset P_{\K}$ is a face 
    of codimension $k$ in $P_{\K}$. 
      The polyhedron $P_{\K}$ together
    with its decomposition into ``faces'' $\{ F_{\sigma} \}_{\sigma\in \K}$ 
    is called a \emph{simple polyhedral complex} (see~\cite[p.428]{DaJan91}).
     \n

   Let $V(\K)$ denote the vertex set of $\K$.
    Any map $\lambda: V(\K)\rightarrow \Z^r$ is
    called a \emph{$\Z^r$-coloring of $\K$}, and any element of $\Z^r$ is called a 
    \emph{color}. For any simplex $\sigma\in \K$,
   \begin{itemize}
     \item let $V(\sigma)$ denote the vertex set of 
   $\sigma$ and,\n
   \item let $G_{\lambda}(\sigma)$     
     denote the toral subgroup of $T^r=(S^1)^r$ corresponding to the subgroup of $\Z^r$
   generated by $\{ \lambda(v) \, |\, v\in V(\sigma) \}$. 
  \end{itemize} 
  
   Given a $\Z^r$-coloring $\lambda$ of $\K$,  we obtain a space $X(\K,\lambda)$ defined by
   \begin{equation}\label{Def:Complex-Quotient}
     X(\K,\lambda):= P_{\K}\times T^r \slash \sim 
   \end{equation}
  where $(p,g)\sim (p',g')$ whenever $p'=p\in F_{\sigma}$ and 
  $g'g^{-1}\in G_{\lambda}(\sigma)$
  for some $\sigma\in \K$.   \n
  
     In particular, if $r=|V(\K)|=m$ and $\{\lambda(v_j)\,; \, 1\leq i \leq m\}$ 
  is a basis of $\Z^m$,
  $X(\K,\lambda)$ is homeomorphic to $\mathcal{Z}_{\K}$.  
  Let $\pi_{\K}: P_{\K}\times T^m \rightarrow \mathcal{Z}_{\K}$ be the corresponding 
 quotient map in~\eqref{Def:Complex-Quotient}. 
 There is a \emph{canonical action} of $T^m$ on $\mathcal{Z}_{\K}$
  defined by:
  \begin{equation} \label{Equ:Canonical-Action}
        g'\cdot \pi_{\K}(p,g) = \pi_{\K} 
        (p,gg'),\ p\in P_{\K},\, g,g'\in T^m.
    \end{equation}
   Then any subgroup of $T^m$ acts canonically on $\mathcal{Z}_{\K}$ through this action.\n
  The following is another way to view the canonical $T^m$-action on
          $\mathcal{Z}_{\K}$. Recall 
       \begin{equation}\label{Equ:Construction-Complex-Moment_1}
     \mathcal{Z}_{\K}=(D^2,S^1)^{\mathcal{K}} = 
        \bigcup_{\sigma\in \mathcal{K}} \big( \prod_{v_j\in \sigma} D^2_{(j)} \times 
           \prod_{v_j\notin \sigma} S^1_{(j)}\big) \subset \prod_{v_j\in [m]} D^2_{(j)} 
    \end{equation}     
  where $D^2_{(j)}$ and $S^1_{(j)}$ are the copy of $D^2$ and $S^1$ associated to $v_j$.
  Notice that $ D^1_{(j)}=S^1_{(j)}*v_j$ (the \emph{join} of $S^1_{(j)}$ with $v_j$). 
  So we can write
    \begin{equation}\label{Equ:Construction-Complex-Moment_2}
     \mathcal{Z}_{\K}=(D^2,S^1)^{\mathcal{K}} = 
        \bigcup_{\sigma\in \mathcal{K}} \big( \prod_{v_j\in \sigma} S^1_{(j)}*v_j \times 
           \prod_{v_j\notin \sigma} S^1_{(j)}\big) 
    \end{equation}     
   We can identify $S^1_{(j)}$ with
  the $j$-th $S^1$-factor in $T^m=(S^1)^m$. 
   Then for any $(g_1,\cdots, g_m)\in T^m$, 
   let $g_j$ act on $S^1_{(j)}$ through
   left translations. This is equivalent to the canonical $T^m$-action
   on $\mathcal{Z}_{\K}$ defined by~\eqref{Equ:Canonical-Action}.
  \n
  
 For any map $\lambda: V(\K)\rightarrow \Z^r$ whose image spans the whole $\Z^r$, 
 we can view the space $X(\K,\lambda)$ in~\eqref{Def:Complex-Quotient} as a quotient space of 
   $\mathcal{Z}_{\K}$ by a toral sbugroup of $T^m$.
   Indeed, let $\{e_1,\cdots, e_m\}$ be a unimodular 
   basis of $\Z^m$ and define a group homomorphism
    \[   \rho_{\lambda}: \Z^m \longrightarrow \Z^r ,\ \rho_{\lambda}(e_j) = \lambda(v_j),\,
        1\leq i \leq m.   \] 
  The kernel of $\rho_{\lambda}$ is a subgroup of $\Z^m$ which determines an $(m-r)$-dimensional 
  toral subgroup $H_{\lambda} \subset T^m$.
         It is easy to see that $X(\K,\lambda)$ is homeomorphic to the quotient space 
          $\mathcal{Z}_{\K}\slash H_{\lambda}$, where $H_{\lambda}$ acts on 
          $\mathcal{Z}_{\K}$ via the canonical $T^m$-action.
       Note that the action of $H_{\lambda}$ on $\mathcal{Z}_{\K}$ 
       is not necessarily free.\n
  
    The cohomology groups of $\mathcal{Z}_{\K}$ can be computed via a Hochster's type
       formula as follows (see~\cite{BasKov02} or~\cite{BP02}). 
       For any subset $J\subset [m]$, let $\K_J$ denote the full 
       subcomplex of $\K$ obtained by restricting to $J$.
       Let $\mathbf{k}$ denote a field or $\Z$ below. We have    
    \begin{equation}\label{Equ:Hoch-complex}
          H^q(\mathcal{Z}_{\K};\mathbf{k}) = \bigoplus_{J\subset [m]} 
       \widetilde{H}^{q-|J|-1}(\K_J;\mathbf{k}), \ q\geq 0. 
     \end{equation}
      where $\widetilde{H}^*(\K_J;\mathbf{k})$ is the reduced cohomology groups of $\K_J$ and 
      $|J|$ denotes the number of elements in $J$. 
      Here we adopt the convention that
      $$\widetilde{H}^{-1}(\K_{\varnothing};\mathbf{k})=\mathbf{k}.$$
    Moreover, it is shown in~\cite{BasKov02} and~\cite{BP02} that there is a natural bigrading on
     $H^*(\mathcal{Z}_{\K};\mathbf{k})$ so that it is
    isomorphic to  $\mathrm{Tor}_{\mathbf{k}[v_1,\cdots, v_m]}(\mathbf{k}[\K];\mathbf{k})$
        as bigraded algebras, where $\mathbf{k}[\K]$ is the face ring (or Stanley-Reisner ring)
         of $\K$ over $\mathbf{k}$ (also see~\cite{Franz06}). \n
        
        The cohomology rings of free quotient spaces of 
        $\mathcal{Z}_{\K}$ can be reprsented in a similar way. 
        Suppose a subtorus $H\subset T^m$ acts freely 
      on $\mathcal{Z}_{\K}$ through the canonical action. It is shown in~\cite{Panov15}
       (also see~\cite[Theorem 7.37]{BP02}) that
      there is a graded algebra isomorphism
      \begin{equation}\label{Equ:Iso-Quotient}
       H^*(\mathcal{Z}_{\K}\slash H ;\mathbf{k}) \cong 
      \mathrm{Tor}_{H^*(B(T^m\slash H);\mathbf{k}} (\mathbf{k}[\K];\mathbf{k})  
      \end{equation}
      where $B(T^m\slash H)$ is the classifying space for the principal $T^m\slash H$-bundle. 
     However, $\mathrm{Tor}_{H^*(B(T^m\slash H);\mathbf{k}} (\mathbf{k}[\K];\mathbf{k})$
     is not so easy to compute in practice and 
     it is not clear whether there exists a Hochster's type formula for
      $H^*(\mathcal{Z}_P\slash H;\mathbf{k})$ as we have
      for $H^*(\mathcal{Z}_P;\mathbf{k})$ in~\eqref{Equ:Hoch-complex}.
        \n

 \begin{rem}
     For
   the calculation of the cohomology ring structure of general polyhedral products, 
   the reader is referred to~\cite{BBCG10, BBCG10-2, Wang-Zheng-13}. 
   \end{rem} \n
     
    An important class of quotient spaces of moment-angle complexes
     are quasitoric manifolds. Let $\K_P$ be the simplicial sphere that is 
     dual to an $n$-dimensional simple convex polytope $P$ with $m$ facets.
    Then $\mathcal{Z}_P = \mathcal{Z}_{\K_P}$ is an
      $(m+n)$-dimensional closed connected manifolds, called
       the \emph{moment-angle manifold} of $P$. Suppose $H\cong T^{m-n}$ is a 
       subgroup of $T^m$ that acts freely on
      $\mathcal{Z}_P$ through the canonical action, the quotient space $\mathcal{Z}_P\slash H$
      is called a \emph{quasitoric manifold} over $P$.
       Quasitoric manifolds are introduced 
       by Davis and Januszkiewicz in~\cite{DaJan91}. 
     \n

      In this paper, we study a special class of quotient spaces of 
      $\mathcal{Z}_{\K}$ and show that their cohomology groups can indeed be computed via
      some Hochster's type formula. These 
        spaces are defined as follows.
       \n
       \begin{itemize}
       \item Let $\upalpha=\{\alpha_1,\cdots, \alpha_k\}$ 
       be a \emph{partition} of the vertex set $V(\K)$
       of a simplicial complex $\K$, i.e. $\alpha_i$'s are disjoint subsets of $V(\K)$
       with $\alpha_1 \cup\cdots \cup \alpha_k=V(\K)$. \n
       
       \item Let $\{ \tilde{e}_1,\cdots, \tilde{e}_k\}$ be a basis of $\Z^k$. We define
      a $\Z^k$-coloring of $\K$, denoted by $\lambda_{\upalpha}$, which
       assigns $\tilde{e}_i$ to all the vertices in $\alpha_i$ ($1\leq i \leq k$).
      \end{itemize}  
            
       Then we obtain a space $X(\K,\lambda_{\upalpha})$ via
        construction~\eqref{Def:Complex-Quotient}, which can be thought of as a 
        quotient space of $\mathcal{Z}_{\K}$ by a rank $m-k$ subtorus 
        $H_{\lambda_{\upalpha}}$ of $T^m$.
       Note that it is possible that the two vertices of a $1$-simplex in $\K$ are assigned 
       the same color (see Figure~\ref{p:Example} for example). 
        \n
          \begin{figure}
         \includegraphics[width=0.85\textwidth]{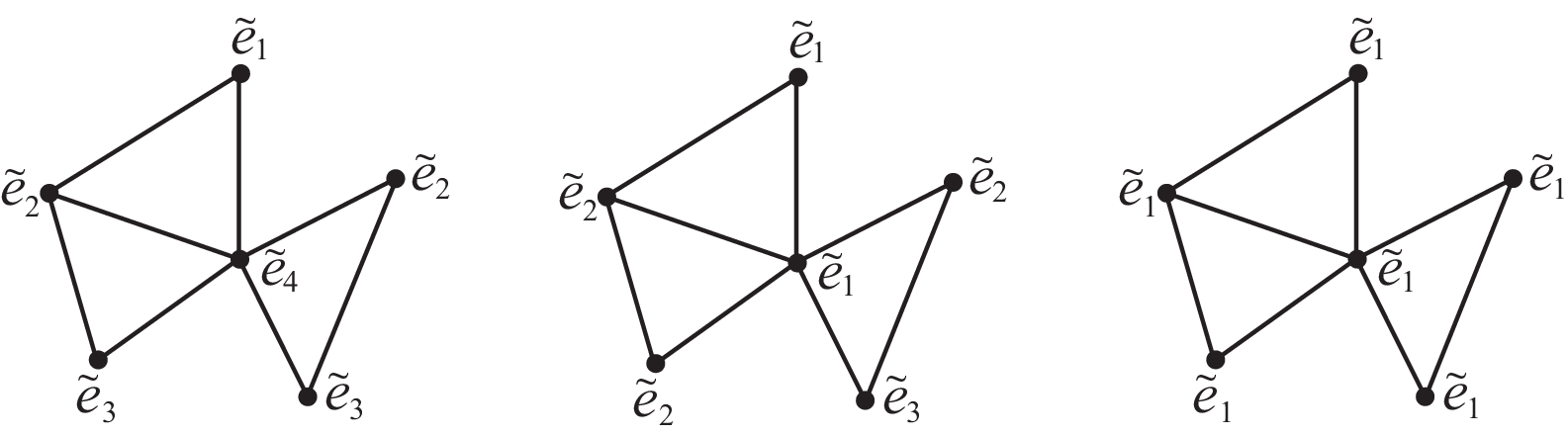}\\
          \caption{Examples of $X(\K,\lambda_{\upalpha})$} 
          \label{p:Example}
      \end{figure}

         Let $\upalpha^*$ denote the \emph{trivial partition}
       of $V(\K)$, i.e. $\upalpha^*=(\alpha_1,\cdots, \alpha_m)$
      where each $\alpha_j=\{ v_j\}$ consists of only one vertex of $\K$. Then according to
       our definition, 
        $$X(\K,\lambda_{\upalpha^*}) = \mathcal{Z}_{\K}.$$
        
       For a non-trivial partition $\upalpha$ of $V(\K)$,    
        the space $X(\K,\lambda_{\upalpha})$ is a priori not a moment-angle complex of any kind.
       But we will see that some topological 
       properties of $X(\K,\lambda_{\upalpha})$ are very similar to
       moment-angle complexes. In particular, the
       cohomology groups of these spaces
       can also be computed by some Hochster's type formula as we do for 
       moment-angle complexes. \n
       
     Let $[k]=\{ 1,\cdots, k\}$. One should keep in mind the difference
     between $[k]$ and the vertex set $[m]$ of $\K$.
     For any simplex $\sigma\in \K$, let
        $$ \I_{\upalpha}(\sigma) := \{ i\in [k] \,
                  ;\, V(\sigma)\cap \alpha_i \neq \varnothing \}
     \subset [k],$$
    which just tells us the set of colors 
    on the vertices of $\sigma$ defined by $\lambda_{\alpha}$.
     Obviously we have
     $0\leq |\I_{\upalpha}(\sigma)| \leq \mathrm{rank}(\sigma)$.
     For any subset $\mathrm{L}\subset [k]$, define
    \begin{equation} \label{Equ:K-Subcomplexes}
        \K_{\upalpha, \mathrm{L}}
      := \text{the subcomplex of $\K$ consisting of $\{ \sigma\in \K\,;\, 
     \I_{\upalpha}(\sigma)\subset \mathrm{L} \}$}.
    \end{equation}

      The main results of this paper are the following two theorems.
      
      \begin{thm}\label{Thm:Main-1}
       Let $\upalpha=\{\alpha_1,\cdots, \alpha_k\}$ 
        be a partition of the vertex set of a finite simplicial complex $\K$. Then we have
        group isomorphisms:
       $$ H^q(X(\K, \lambda_{\upalpha});\mathbf{k}) \cong 
          \underset{\mathrm{L}\subset [k]}{\bigoplus} \widetilde{H}^{q-|\mathrm{L}|-1} 
          (\K_{\upalpha, \mathrm{L}};\mathbf{k}), \, \forall q\geq 0.$$ 
      \end{thm}  
   
     According to the isomorphism in~\eqref{Equ:Iso-Quotient},
     the above formula gives us a way to compute 
     $\mathrm{Tor}_{H^*(B(T^m\slash H);\mathbf{k}} (\mathbf{k}[\K];\mathbf{k})$ in
     the setting of Theorem~\ref{Thm:Main-1}. 
      Note that the above formula for $\upalpha^*$
    gives~\eqref{Equ:Hoch-complex}. \n
     
      In addition,
    it was shown in~\cite[Corollary 2.23]{BBCG10}
       that the Hochster's formula for the
      cohomology groups of $\mathcal{Z}_{\K}$ follows from a 
      stable decomposition of $\mathcal{Z}_{\K}$. We have parallel 
      results for $X(\K,\lambda_{\upalpha})$ as well.\n
      
       \begin{thm}\label{Thm:Stable-Decomp}
      Let $\upalpha=\{\alpha_1,\cdots, \alpha_k\}$ 
        be a partition of the vertex set of a finite simplicial complex $\K$.
          There are homotopy equivalences:
         $$  
           \mathbf{\Sigma}( X(\K,\lambda_{\upalpha}) ) \simeq \bigvee_{\mathrm{L}\subset [k]} 
           \mathbf{\Sigma}^{|\mathrm{L}|+2}( \K_{\upalpha,\mathrm{L}})$$
          where the bold $\mathbf{\Sigma}$ denotes the reduced suspension. 
       \end{thm}
         
      The paper is organized as follows. In section 2, we construct some
      natural cell decomposition of $X(\K,\lambda_{\upalpha})$ and use 
      it to
       compute the the cohomology groups of $X(\K,\lambda_{\upalpha})$, which leads to
       a proof of
        Theorem~\ref{Thm:Main-1}. In section 3, we use
        the same strategy in~\cite{BBCG10} to study the stable decompositions of
         $X(\K,\lambda_{\upalpha})$ and give a proof of Theorem~\ref{Thm:Stable-Decomp}.
         In section 4, we show that the
         moment-angle complex of any finite 
         simplicial poset $\mathcal{S}$ is homotopy equivalent to
        $X(\K,\lambda_{\upalpha})$ for some finite 
        simplicial complex $\mathcal{K}$ and a partition
        $\upalpha$ of $V(\mathcal{K})$.
       In section 5, we generalize our results on $X(\K,\lambda_{\upalpha})$
       to a wider range of spaces.\\

     \section{Cohomology groups of $X(\K,\lambda_{\upalpha})$}
   
     Suppose the vertex set of $\K$ is $[m]=\{ v_1,\cdots, v_m\}$. Let 
     $\Delta^{[m]}$ be the simplex with vertex set $[m]$. For a partition
     $\upalpha=\{ \alpha_1,\cdots, \alpha_k \}$ of $[m]$, let
     $\Delta^{\alpha_i}$ denote the face of $\Delta^{[m]}$ whose 
     vertex set is $\alpha_i$. Then $\K$ can be thought of as 
      a simplicial subcomplex of $\Delta^{[m]}$. 
      Next, we construct a cell decomposition of $X(\K,\lambda_{\upalpha})$.\nn

   \subsection{The cell decomposition of $X(\K,\lambda_{\upalpha})$}\ \n
   
     According to the construction of $X(\K,\lambda_{\upalpha})$ in~\eqref{Def:Complex-Quotient}, 
     it is easy to see that $X(\K,\lambda_{\upalpha})$ is homeomorphic to
        the quotient space
       of $\mathcal{Z}_{\K}$ by the canonical action of the 
       toral subgroup $H_{\lambda_{\upalpha}}$ of $T^m$ corresponding to
       the subgroup of $\Z^m = \langle e_1,\cdots, e_m\rangle$ generated by the set
       $$\{ e_{j} - e_{j'} \,|\, v_{j}, v_{j'} \in \upalpha_i \ \text{for some}\ 
            1\leq i \leq k \}  \subset \Z^m.$$
      In other words,
      the action $H_{\lambda_{\upalpha}}$ on $\mathcal{Z}_{\K}=P_{\K}\times T^m \slash \sim$ identifies the
      $S^1_{(j)}$ and $S^1_{(j')}$ in $T^m=S^1_{(1)}\times \cdots \times S^1_{(m)}$
       whenever $v_j$ and $v_{j'}$ belong
       to the same $\alpha_i$.
       Considering the partition $\upalpha$ of the vertex set of $\K$,
        we can rewrite the decomposition of $\mathcal{Z}_{\K}$ 
      in~\eqref{Equ:Construction-Complex-Moment_2} as:
       \begin{equation}\label{Equ:Construction-Complex-Moment_3}
     \mathcal{Z}_{\K}= 
        \bigcup_{\sigma\in \mathcal{K}} \Big(  \big( \prod_{i\in \mathrm{I}_{\upalpha}(\sigma)} 
        \prod_{v_j\in V(\sigma)\cap \alpha_i} S^1_{(j)}*v_j \big)\times 
           \prod_{v_j\notin \sigma} S^1_{(j)}\Big).
    \end{equation}  
    Then with respect to this decomposition of $ \mathcal{Z}_{\K}$,
    we obtain a decomposition of 
    $X(\K,\lambda_{\upalpha})$ by Lemma~\ref{Lem:Quotient-Join} below.
       \begin{align}\label{Equ:Decomp-Quotient}
       X(\K,\lambda_{\upalpha}) 
       &= \bigcup_{\sigma\in \K} \Big(
       \prod_{i\in \mathrm{I}_{\upalpha}(\sigma)} 
       \big(S^1_{(i)} * \underset{v\in V(\sigma)\cap \alpha_i}{\Asterisk} v \big)\times
        \prod_{i\in [k]\backslash \mathrm{I}_{\upalpha}(\sigma)} S^1_{(i)} \Big) \notag \\
       &=  \bigcup_{\sigma\in \K} \Big(
       \prod_{i\in \mathrm{I}_{\upalpha}(\sigma)} S^1_{(i)}*(\sigma\cap \Delta^{\alpha_i}) 
        \times \prod_{i\in [k]\backslash \mathrm{I}_{\upalpha}(\sigma)} S^1_{(i)} \Big)
         \subset  \prod_{i\in [k]} S^1_{(i)}* \Delta^{\alpha_i} 
       \end{align} 
      where $S^1_{(i)}$ is a copy of $S^1$ corresponding to $i\in [k]$, which can be considered
      as the join of $S^1$ with the empty face of $\Delta^{\alpha_i}$. \n
            
     \begin{lem}\label{Lem:Quotient-Join}
     If we identify all the $S^1$ factors in a product 
      $(S^1*v_1)\times\cdots \times(S^1*v_s)$, 
      the quotient space $(S^1*v_1)\times\cdots \times(S^1*v_s)\slash \sim$ is homeomorphic to
      $S^1*(v_1*\cdots * v_s)$. Note that $v_1*\cdots * v_s$ can be identified with a
      simplex whose vertex set is $\{v_1,\cdots, v_s\}$. 
    \end{lem}  
    \begin{proof}
     The points in $(S^1*v_1)\times\cdots \times(S^1*v_s)$ can be written as
     \[ \big( (t_1v_1 + (1-t_1)x_1),\cdots, (t_sv_s + (1-t_s)x_s) \big), \, x_i\in S^1,
      0\leq t_i \leq 1, 1\leq i \leq s. \]
    In we identify all the $S^1$ factors in the above product, the points in the quotient space
    can be written as 
    \[ P_{t_1,\cdots, t_s,x} = \big( (t_1v_1 + (1-t_1)x),\cdots, 
       (t_sv_s + (1-t_s)x) \big), \, x\in S^1,
      0\leq t_i \leq 1, 1\leq i \leq s. \] 
    Then mapping any $P_{t_1,\cdots, t_s,x}$ to the point $\hat{P}_{t_1,\cdots, t_s,x}$ 
    in $S^1*(v_1*\cdots * v_s)$ below
      \[ \hat{P}_{t_1,\cdots, t_s,x} =  \frac{t_1}{s}v_1 +\cdots +\frac{t_s}{s}v_s +
          \frac{1-t_1-\cdots-t_s}{s} x \]
    defines a homeomorphism from $(S^1*v_1)\times\cdots \times(S^1*v_s)\slash \sim$ 
    to $S^1*(v_1*\cdots * v_s)$.
    \end{proof}
    
      \begin{rem}
    We see from~\eqref{Equ:Decomp-Quotient} that the building blocks of 
    $X(\K,\lambda_{\upalpha}) $ 
    are spaces obtained by mixtures of
     Cartesian products and joins of 
    some simple spaces (i.e. points and $S^1$). 
    The building blocks of polyhedral products $(X,A)^{\mathcal{K}}$, however, only involve
    Cartesian products of spaces. 
    In addition, we have \emph{polyhedral join} (see~\cite{Anton13}) and
    \emph{polyhedral smash product} (see~\cite{BBCG10})
     whose building blocks only involve joins and smash products, respectively.
   It should be interesting to study
     spaces whose building blocks involve mixtures of 
     Cartesian products, joins and smash products. 
 \end{rem}

    To obtain a cell decomposition of $X(\K,\lambda_{\upalpha})$ 
    from~\eqref{Equ:Decomp-Quotient}, we need to choose a cell
     decomposition of the torus $T^k$. 
   First of all, a circle $S^1=\{ z\in\mathbb{C}\,; |z|=1 \}$ has a natural
     cell decomposition $\{ e^0,e^1\}$ where $e^0=\{1\}\in S^1$ and $e^1=S^1\backslash e^0$.
    We consider $T^k$ as the product $\prod^k_{i=1}S^1_{(i)}$ and
     equip $T^k$ with the product cell structure (see~\cite[3.B]{Hatcher}).
    Then the cells in $T^k$ can be indexed by subsets of $[k]=\{ 1,\cdots, k\}$. 
     More specifically, any subset 
     $\mathrm{L}\subset [k]$ determines a unique cell $U_{\mathrm{L}}$ in $T^k$ where
     \[ U_{\mathrm{L}} =  \prod_{i\in \mathrm{L}} e^1_{(i)} \times 
       \prod_{i\in [k]\backslash \mathrm{L}} e^0_{(i)},\ 
     \dim(U_{\mathrm{L}})=|\mathrm{L}|. \]
     Here $ e^0_{(i)}, e^1_{(i)}$ denote the cells in $S^1_{(i)}$ for each
     $i\in [k]$. \n
     
     Observe that for any $\sigma\in \K$, 
      $\prod_{i\in \mathrm{I}_{\upalpha}(\sigma)} S^1_{(i)}*(\sigma\cap \Delta^{\alpha_i})$
      is homeomorphic to a closed ball of dimension 
      $\mathrm{rank}(\sigma)+|\I_{\upalpha}(\sigma)|$. Now for each $\mathrm{L} \subset [k]\backslash \mathrm{I}_{\upalpha}(\sigma)$, define
       $$ B_{(\sigma,\mathrm{L})} := \text{the relative interior of}\ 
    \Big( \prod_{i\in \mathrm{I}_{\upalpha}(\sigma)} 
        S^1_{(i)}*(\sigma\cap \Delta^{\alpha_i}) \Big)\times U_{\mathrm{L}}.$$  
      Then from~\eqref{Equ:Decomp-Quotient}, 
      a cell decomposition of $X(\K,\lambda_{\upalpha})$ is given by:
    \begin{equation} \label{Equ:Cell-Decomp-Complex}
       \mathscr{B}_{\upalpha}(\K):=\{ B_{(\sigma,\mathrm{L})} \, |\, \sigma\in \K,\
    \mathrm{L} \subset [k]\backslash \mathrm{I}_{\upalpha}(\sigma)\}    
   \end{equation} 
      Note that $B_{(\sigma,\mathrm{L})}$
      is an open cell of dimension
    $\mathrm{rank}(\sigma)+|\I_{\upalpha}(\sigma)|+|\mathrm{L}|$. \nn

   \subsection{The cochain complex of $X(\K,\lambda_{\upalpha})$}
  
   \ \n
   
   For any coefficient ring $\mathbf{k}$, 
   let $C^*(X(\K,\lambda_{\upalpha});\mathbf{k})$
       be the 
     cellular cochain complex corresponding to the
      cell decomposition $\mathscr{B}_{\upalpha}(\K)$.
      If we want to write the boundary maps of the cochains
      in $C^*(X(\K,\lambda_{\upalpha});\mathbf{k})$,
       we need to put orientations on
        the base elements.  To do this, we need to first assign
         orientations to all the simplicies of $\K$. 
       For convenience, we put a total 
      ordering $\prec$ on the vertex set 
      $\{v_1,\cdots, v_m\}$ of $\Delta^{[m]}$ so that they appear in the increasing 
      order in $\alpha_1$ until $\alpha_k$.
      In other words, for any $1\leq i < k$ all the vertices of $\Delta^{\alpha_i}$ have less
      order than the vertices in $\Delta^{\alpha_{i+1}}$. 
      Moreover, the vertex-ordering of $\Delta^{[m]}$ induces a vertex-ordering
      of any simplex $\omega\in \Delta^{[m]}$, 
       which determines an orientation of $\omega$. Then the boundary of
       $\omega$ is
     \begin{equation}\label{Equ:Sign-Boundary}
        \partial \omega =\underset{\dim(\sigma)=\dim(\omega)-1}{\sum_{\sigma\subset \omega}}
          \varepsilon(\sigma,\omega)\, \sigma . 
      \end{equation}    
      Here if $V(\omega)=V(\sigma)\cup \{ v\}$, then
    $\varepsilon(\sigma,\omega)$ is equal to 
    $(-1)^{l(v,\omega)}$ where $l(v,\omega)$ is the number 
    vertices of $\omega$ that are less than $v$ with respect to the vertex-ordering $\prec$.\n
      
         The following definition is very useful for us to simplify our argument later.
         
    \begin{defi}[Simplex with a ghost face] \label{Def:Simplex-Ghost-Face}
    For any $m\geq 1$, let $\hat{0}$ denote the empty face of 
    $\Delta^{[m]}$ (distinguished from the empty simplex
    $\hat{\mathbf{0}}$ in $\K$).
     In addition, we attach a
      \emph{ghost face} $-\hat{1}$ to $\Delta^{[m]}$ with 
     the following conventions.
      \begin{itemize}        
        \item $\dim(\hat{0})=\dim(-\hat{1})= -1$, $\mathrm{rank}(\hat{0})=
      \mathrm{rank}(-\hat{1}) = 0$. \n
      
        \item The interiors of $\hat{0}$ and $-\hat{1}$ are themselves. \n
        
        \item The boundary of any vertex of $\Delta^{[m]}$ is $\hat{0}$.\n
        
        \item The boundaries of $\hat{0}$ and $-\hat{1}$ are empty, \n
      \end{itemize}
     In the rest of the paper we use $\widehat{\Delta}^{[m]}$ to 
     denote $\Delta^{[m]}$ with the ghost 
     face $-\hat{1}$.
     
     \end{defi}

 Let $\{ e^0, e^1\}$ be the cell decomposition of $S^1$ where
 $\dim(e^0)=0$, $\dim(e^1)=1$, and $e^0$, $e^1$ are both oriented.
  Then given an orientation of each face of $\Delta^{[m]}$,
     we obtain an oriented cell decomposition of $S^1*\Delta^{[m]}$ by
     $$\{e^0\},\ \{e^1\}, 
      \ \{ S^1 * \sigma^{\circ}\,|\, \sigma \ \text{is a nonempty simplex in}\ \Delta^{[m]} 
       \}.$$ 
        If we formally define $S^1*\hat{0} = e^1$ and
     $S^1* -\hat{1} = e^0$, we can write a basis of 
     the cellular chain complex $C_*(S^1*\Delta^{[m]};\mathbf{k})$
    as $ \{ S^1*\sigma^{\circ} \, |\, \sigma \in \widehat{\Delta}^{[m]}\}$ where
    the orientation of $S^1*\sigma^{\circ}$ 
    is determined canonically by the orientations of $S^1$ and
    $\sigma$. \n
       
    Let $\{ y^{\sigma}\, |\, \sigma \in \widehat{\Delta}^{[m]} \}$ be a basis 
     for the cellular cochain complex $C^*(S^1*\Delta^{[m]};\mathbf{k})$, 
     where $y^{\sigma}$ is the dual of $S^1*\sigma^{\circ}$. 
      For any nonempty simplex $\sigma$ in $\Delta^{[m]}$, 
     \begin{equation} \label{Equ:Diff-Complex-0}
         d(y^{\sigma}) = \underset{\dim(\tau)=\dim(\sigma)+1}{\sum_{\sigma\subset \tau}} 
               \varepsilon{(\sigma,\tau})\cdot y^{\tau}.
      \end{equation}        
       In addition, we have
     \begin{equation} \label{Equ:Diff-Complex-0-0}
        d(y^{-\hat{1}})=0,\ \ d(y^{\hat{0}}) = \sum_{v\in [m]} y^{v}. 
      \end{equation}  
      Note that $y^{-\hat{1}}$ and $y^{\hat{0}}$ are cochains in dimension $0$ and $1$, 
      respectively.\n
     
      By setting the $\sigma$ in~\eqref{Equ:Decomp-Quotient} to the big simplex $\Delta^{[m]}$,
      we obtain a homeomorphism  
      $$X(\Delta^{[m]},\lambda_{\upalpha})
      \cong \prod_{i\in [k]} S^1*\Delta^{\alpha_i}.$$ 
      Let $X(\Delta^{[m]},\lambda_{\upalpha})$ be equipped with the 
      product cell structure of each
      $S^1*\Delta^{\alpha_i}$.       
      The corresponding cellular cochain complex $C^*(X(\Delta^{[m]},\lambda_{\upalpha});
      \mathbf{k})$ has a basis
       $$\{ \mathbf{y}^{\Phi} = y^{\sigma_1}\times \cdots \times y^{\sigma_k}; \   
       \Phi =(\sigma_1,\cdots, \sigma_k), \ \text{where}\ 
       \sigma_i\in \widehat{\Delta}^{\alpha_i},\, 1\leq i \leq k \}.$$ 
    We need to introduce two more notations for our argument below.               
    \begin{itemize}
    \item For any vertex $v$ of $\K$ and a partition $\upalpha=\{ \alpha_1,\cdots,\alpha_k\}$
     of $V(\K)$,
     let $i_{\upalpha}(v) \in [k]$ denote the index so that $v$ belongs to
     the subset $\alpha_{i_{\upalpha}(v)}$.\n 
    
   \item  For any $i\in \mathrm{L} \subset [k]$, define
     $\kappa(i,\mathrm{L})= (-1)^{r(i,\mathrm{L})}$, where
     $r(i,\mathrm{L})$ is the number of elements in $\mathrm{L}$ less than $i$. 
    Moreover, for any simplex $\sigma \in \K_{\upalpha, \mathrm{L}}$ we define               
    \begin{equation}\label{Equ:Kappa}
       \kappa(\sigma, \mathrm{L}) := \prod_{v\in V(\sigma)}
     \kappa(i_{\upalpha}(v), \mathrm{L}). 
   \end{equation}  
    \begin{equation} \label{Equ:Sign-Relation-1}
      \text{ So if $V(\omega)=V(\sigma)\cup \{v \}$, we have
    $\kappa(\omega, \mathrm{L})=\kappa(\sigma, \mathrm{L}) \cdot
      \kappa(i_{\upalpha}(v), \mathrm{L})$.}
    \end{equation}  
 \end{itemize}
     \n
     
  The differential of $\mathbf{y}^{\Phi}$ in $C^*(X(\Delta^{[m]},\lambda_{\upalpha});
      \mathbf{k})$ is given by:   
       \begin{equation} \label{Equ:Diff-complex-1}
      d(\mathbf{y}^{\Phi}) := \sum_{1\leq i \leq k} 
       \iota(\Phi,\sigma_i)\, y^{\sigma_1}\times \cdots \times d y^{\sigma_i} 
        \times \cdots \times y^{\sigma_k}, 
        \end{equation}
       where $\iota(\Phi,\sigma_i) =
        (-1)^{\sum^{i-1}_{l=1}\dim(y^{\sigma_l})}$.
         For a simplex $\sigma\in \Delta^{[m]}$ and $\J\subset [k]\backslash \I_{\upalpha}(\sigma)$,
     let 
     $$\Phi^{\J}_{\sigma}=(\sigma^{\J}_1,\cdots, \sigma^{\J}_k)$$
     \begin{equation} \label{Equ:Phi-sigma-J}
       \text{where}\ \,   \sigma^{\J}_i = \begin{cases}
          \sigma\cap \Delta^{\alpha_i},   &  
                        \text{ $i\in \I_{\upalpha}(\sigma)$; } \\
          \ \hat{0}\in \widehat{\Delta}^{\alpha_i},     &   \text{ $i\in \J$; } \\
          -\hat{1}\in \widehat{\Delta}^{\alpha_i},   &   \text{ $i\in  [k]\backslash 
              (\I_{\upalpha}(\sigma)\cup \J)$;} 
              \end{cases} \ 
           \end{equation} 
 So by our definition when $\Phi=\Phi^{\J}_{\sigma}=(\sigma^{\J}_1,\cdots, \sigma^{\J}_k)$,
 we have 
         $$\dim(y^{\sigma^{\J}_i})=
              \begin{cases}
          \mathrm{rank}(\sigma^{\J}_i)+1,   &   \text{ $i\in \I_{\upalpha}(\sigma)$; } \\
            \ 1,     &   \text{ $i\in \J$; } \\
          \  0,   &   \text{ $i\in  [k]\backslash 
              (\I_{\upalpha}(\sigma)\cup \J)$.} 
              \end{cases} 
         $$
 For $\Phi=\Phi^{\J}_{\sigma}$, the formula~\eqref{Equ:Diff-complex-1} reads:
     \begin{align} 
         d (\mathbf{y}^{\Phi^{\J}_{\sigma}})
           & \ = \sum_{1\leq i \leq k}  \iota(\Phi^{\J}_{\sigma},\sigma^{\J}_i)\, 
                y^{\sigma^{\J}_1}\times \cdots \times d y^{\sigma^{\J}_i} \times \cdots \times 
                y^{\sigma^{\J}_k} \notag \\ 
           & \overset{\eqref{Equ:Diff-Complex-0}}{=} 
           \sum_{1\leq i \leq k} \iota(\Phi^{\J}_{\sigma},\sigma^{\J}_i)\, 
                y^{\sigma^{\J}_1}\times \cdots \times \Big( 
                  \underset{\dim(\tau)=\dim(\sigma)+1}{\sum_{\sigma^{\J}_i\subset \tau \subset
                  \Delta^{\alpha_i}}} 
               \varepsilon{(\sigma^{\J}_i,\tau})\cdot y^{\tau} \Big) \times \cdots \times 
                y^{\sigma^{\J}_k} . \notag 
         \end{align}   
 Note that if $V(\omega)=V(\sigma)\cup \{ v\}$ where
 $v\in \alpha_i$, then $\omega^{\J}_i=\sigma^{\J}_i\cup \{ v\}$ and we have
  $$\kappa(i, \I_{\upalpha}(\sigma)\cup\J) \cdot 
    \varepsilon(\sigma,\omega)= \iota(\Phi^{\J}_{\sigma},\sigma^{\J}_i)\cdot
      \varepsilon(\sigma^{\J}_i,\tau).
  $$   
  This is because for each $i\in  [k]\backslash 
              (\I_{\upalpha}(\sigma)\cup \J)$, $\dim(y^{\sigma^{\J}_i}) = 0$. So we obtain
 \begin{equation} \label{Equ:Diff-complex-2}
      d (\mathbf{y}^{\Phi^{\J}_{\sigma}}) =        
            \underset{\dim(\omega)=\dim(\sigma)+1} 
            {\sum_{\sigma\subset \omega,\, \mathrm{I}_{\upalpha}(\omega)
                  \subset \mathrm{I}_{\upalpha}(\sigma)\cup \mathrm{J}}} 
             \kappa\big( i_{\upalpha}(\omega\backslash \sigma), \I_{\upalpha}(\sigma)\cup\J \big) 
             \cdot  \varepsilon(\sigma,\omega) \,
                 \mathbf{y}^{\Phi_{\omega}^{\mathrm{I}_{\upalpha}(\sigma)\cup \mathrm{J}\backslash 
                      \mathrm{I}_{\upalpha}(\omega)}}  
   \end{equation}                   
  where $\omega\backslash \sigma$ denote the only vertex of $\omega$ that is 
  not in $\sigma$.\nn
  
          For the simplicial complex $\K \subset \Delta^{[m]}$ and 
          any simplex $\sigma\in \K$, $\mathbf{y}^{\Phi^{\J}_{\sigma}}$ is
   a cochain of dimension $\mathrm{rank}(\sigma)+|\I_{\upalpha}(\sigma)|+|\J|$
    in $C^*(X(\K,\lambda_{\upalpha});\mathbf{k})$
  that is dual to the cell $B_{(\sigma,\J)}$
  (see~\eqref{Equ:Cell-Decomp-Complex}). 
     So $C^*(X(\K,\lambda_{\upalpha});\Z)$ has a basis 
      $$\{  \mathbf{y}^{\Phi^{\J}_{\sigma}}\,;\, \sigma\in\K,\ 
       \J\subset [k]\backslash \I_{\upalpha}(\sigma)\}.$$                  
    Note that for any $\sigma\in \K$, the differential $d (\mathbf{y}^{\Phi^{\J}_{\sigma}})$ 
    in $C^*(X(\K,\lambda_{\upalpha});\mathbf{k})$
     is only the sum of those terms 
     $\mathbf{y}^{\Phi_{\omega}^{\mathrm{I}_{\upalpha}(\sigma)\cup \mathrm{J}\backslash 
                      \mathrm{I}_{\upalpha}(\omega)}}$ on the right hand side 
                      of~\eqref{Equ:Diff-complex-2} with $\omega\in \K$.
                      \nn

   \begin{proof}[\textbf{Proof of Theorem~\ref{Thm:Main-1}}]\ \n
   
      For any subset $\mathrm{L}\subset [k]$, let 
      $C^{*,\mathrm{L}}(X(\K,\lambda_{\upalpha});\mathbf{k})$ denote the $\mathbf{k}$-submodule of
      $C^*(X(\K,\lambda_{\upalpha});\mathbf{k})$ generated by the following set  
          $$\{ B_{(\sigma, \mathrm{J})} \,|\, \mathrm{I}_{\upalpha}(\sigma)\cup 
            \mathrm{J} =\mathrm{L},\, \sigma\in\K,\, 
       \J\subset [k]\backslash \I_{\upalpha}(\sigma)\}.$$
       
   From the differential of $C^*(X(\K,\lambda_{\upalpha});\mathbf{k})$ 
   in~\eqref{Equ:Diff-complex-2}, we see that 
   $C^{*,\mathrm{L}}(X(\K,\lambda_{\upalpha});\mathbf{k})$ is actually a cochain subcomplex of
     $C^*(X(\K,\lambda_{\upalpha});\mathbf{k})$.
     We denote its cohomology groups by
     $H^{*,\mathrm{L}}(X(\K,\lambda_{\upalpha});\mathbf{k})$.      
      Then we have the following decompositions:
    \begin{equation} \label{Equ:Decomp-Complex}  
          H^*(X(\K,\lambda_{\upalpha});\mathbf{k})=\bigoplus_{\mathrm{L}\subset [k]}
        H^{*,\mathrm{L}}(X(\K,\lambda_{\upalpha});\mathbf{k}).
       \end{equation}
   For any subset $\mathrm{L}\subset [k]$, let $C^*(\K_{\upalpha, \mathrm{L}};\mathbf{k})$ 
   denote the simplicial cochain complex of $\K_{\upalpha, \mathrm{L}}$. 
   For any simplex $\sigma\in \K$, let
   $\sigma^*$ be the cochain dual to $\sigma$ in 
   $C^*(\K_{\upalpha, \mathrm{L}};\mathbf{k})$.
   Then we have the following isomorphism of $\mathbf{k}$-modules:
   \begin{align*} 
     \varphi^{\mathrm{L}}_{\upalpha}: C^*(\K_{\upalpha, \mathrm{L}};\mathbf{k}) 
       & \longrightarrow  C^{*,\mathrm{L}}(X(\K,\lambda_{\upalpha});\mathbf{k}) \\
     \sigma^*  \ & \longmapsto \   \kappa(\sigma, \mathrm{L})
     \, \mathbf{y}^{\Phi_{\sigma}^{\mathrm{L}\backslash \I_{\upalpha}(\sigma)}}
    \end{align*}
  Note
  $
    \dim(\mathbf{y}^{\Phi_{\sigma}^{\mathrm{L}\backslash \I_{\upalpha}(\sigma)}})=
     \mathrm{rank}(\sigma) + |\mathrm{L}|$.
  Moreover, $\varphi^{\mathrm{L}}_{\upalpha}$
    is actually a chain complex isomorphism. Indeed by the differential of 
    $C^*(X(\K,\lambda_{\upalpha});\mathbf{k})$ shown in~\eqref{Equ:Diff-complex-2}, 
        \begin{align*}
     d (\varphi^{\mathrm{L}}_{\upalpha}(\sigma^*)) & = 
    d\big( \kappa(\sigma, \mathrm{L})
     \, \mathbf{y}^{\Phi_{\sigma}^{\mathrm{L}\backslash \I_{\upalpha}(\sigma)}} \big) \\
     & =  \kappa(\sigma, \mathrm{L})    \underset{\dim(\omega)=\dim(\sigma)+1} 
            {\sum_{\sigma\subset \omega\in \K,\, \mathrm{I}_{\upalpha}(\omega)
                  \subset \mathrm{L}}} 
              \kappa( i_{\upalpha}(\omega\backslash \sigma), \mathrm{L}) 
              \cdot  \varepsilon(\sigma,\omega) \,
                 \mathbf{y}^{\Phi_{\omega}^{\mathrm{L}\backslash 
                      \mathrm{I}_{\upalpha}(\omega)}}  \\
     & =    \underset{\dim(\omega)=\dim(\sigma)+1} 
            {\sum_{\sigma\subset \omega\in \K,\, \mathrm{I}_{\upalpha}(\omega)
                  \subset \mathrm{L}}} 
             \kappa(\omega, \mathrm{L})  \cdot  \varepsilon(\sigma,\omega) \,
                 \mathbf{y}^{\Phi_{\omega}^{\mathrm{L}\backslash 
                      \mathrm{I}_{\upalpha}(\omega)}} 
     = \varphi^{\mathrm{L}}_{\upalpha}(d\sigma^*).        
   \end{align*}            
    The third ``$=$'' uses the relation
     $\kappa(\omega, \mathrm{L})=\kappa(\sigma, \mathrm{L})
      \cdot \kappa( i_{\upalpha}(\omega\backslash \sigma), \mathrm{L})$ 
      (see~\eqref{Equ:Sign-Relation-1}).   
    So we have an additive isomorphism of cohomology groups
    \[ 
     \widetilde{H}^q(\K_{\upalpha, \mathrm{L}};\mathbf{k}) \cong  
                        H^{q+|\mathrm{L}|+1,\mathrm{L}}(X(\K,\lambda_{\upalpha});\mathbf{k}).
        \]
     $$ \text{Let}\ \ \varphi_{\upalpha} = \underset{\mathrm{L}\subset [k]}{\bigoplus}
     \varphi^{\mathrm{L}}_{\upalpha} :\
      \underset{\mathrm{L}\subset [k]}{\bigoplus} 
                   \widetilde{H}^{q-|\mathrm{L}|-1}(\K_{\upalpha, \mathrm{L}};\mathbf{k}) \longrightarrow
                    H^q(X(\K,\lambda_{\upalpha});\mathbf{k}).$$      
   Then $\varphi_{\upalpha}$ is an isomorphism that satisfies our requirement.                               
   \end{proof}                   
  
   \begin{exam}[Balanced Simplicial Complex and Pullbacks from the Linear Model] 
     \ \ \
      An $(n-1)$-dimensional simplicial complex $\K$ is called \emph{balanced} if 
    there exists a map $\phi : V(\K) \rightarrow [n]=\{1,\cdots, n\}$ such that
    if $\{ v,v'\}$ is an edge of $\K$, then $\phi(x)\neq \phi(y)$ (see~\cite{St96}). 
    We call $\phi$
    an \emph{$n$-coloring} on $\K$. 
     It is easy to see that $\K$ is balanced if and only if $\K$ admits 
     a non-degenerate simplicial map onto $\Delta^{n-1}$. In fact, if we identify
     the vertex set of $\Delta^{n-1}$ with $[n]$, any $n$-coloring $\phi$ on $\K$
     induces a non-degenerate simplicial map from $\K$ to $\Delta^{n-1}$ which
     sends a simplex $\sigma\in \K$ with $V(\sigma)=\{ v_{i_1},\cdots, v_{i_s}\}$
     to the face of $\Delta^{n-1}$ spanned by $\{ \phi(v_{i_1}),\cdots, \phi(v_{i_s}) \}$.
       \n
    
    Suppose
    $\phi:V(\K)\rightarrow [n]$ is an $n$-coloring of an $(n-1)$-dimensional simplicial complex
     $\K$.
    Let $\{e_1,\cdots, e_n\}$ be a basis of $\Z^n$.
    Then $\phi$ uniquely determines a $\Z^n$-coloring 
    $\lambda^{\phi}: V(\K)\rightarrow \Z^n$ where $\lambda^{\phi}(v)= e_{\phi(v)}$.
   The space $X(\K,\lambda^{\phi})$ is called a \emph{pullback from the linear model} 
    in~\cite[Example 1.15]{DaJan91}.      
        On the other hand, we have a partition of $V(\K)$
       defined by $\upalpha_{\phi} := \{ \phi^{-1}(1),\cdots \phi^{-1}(n) \}$. 
   By our notation in section 1, we have
     $X(\K,\lambda^{\phi}) =X(\K,\lambda_{\upalpha_{\phi}})$. 
  Then by Theorem~\ref{Thm:Main-1}, the cohomology groups of $X(\K,\lambda^{\phi})$ 
   can be computed by
     \[  H^q(X(\K,\lambda^{\phi});\mathbf{k}) \cong \bigoplus_{\mathrm{L}
           \subset [n]}  \widetilde{H}^{q-|\mathrm{L}|-1}
                (\K_{\upalpha_{\phi}, \mathrm{L}};\mathbf{k}),\ \forall q\geq 0. \]
   \end{exam}
   
    \ \\
   
   \section{Stable decompositions of $X(\K,\lambda_{\upalpha})$}
   
   It is shown in~\cite{BBCG10} that the stable homotopy type of a polyhedral
   product $(\underline{\mathbb{X}},\underline{\mathbb{A}})^{\K}$  is a wedge of
   spaces, which implies corresponding homological decompositions of
   $(\underline{\mathbb{X}},\underline{\mathbb{A}})^{\K}$. In this section, we
    prove a parallel result
   for $X(\K,\lambda_{\upalpha})$. Our argument proceeds along the same line as~\cite{BBCG10}.\n
   
     Let
    $\upalpha=\{ \alpha_1,\cdots,\alpha_k \}$ be a
     partition of the vertex set $V(\K)$ of $\K$.
     For any $i\in [k]$, choose the base-point of $S^1_{(i)}= e^0_{(i)} \cup e^1_{(i)}$ to be
     $e^0_{(i)}$. So
    \begin{itemize}
      \item $e^0_{(i)}$ is a base-point of $S^1_{(i)}*\tau$ 
   for any simplex $\tau\in\K$, and\n
   
   \item $(e^0_{(i)},\cdots, e^0_{(i_s)})$ is
    a base-point of $S^1_{(i_1)}\times\cdots\times S^1_{(i_s)}$ for any $i_1,\cdots, i_s\in [k]$.
   \end{itemize}
  Then for any subset $\mathrm{L} \subset [k]$
     and any simplex $\sigma\in\K$, it is meaningful to define
    \[ \mathbf{W}^{S^1}_{\upalpha, \mathrm{L}}(\sigma) := 
    \bigwedge_{i\in \mathrm{I}_{\upalpha}(\sigma)\cap \mathrm{L}} 
    S^1_{(i)}*(\sigma\cap \Delta^{\alpha_i}) 
    \wedge \bigwedge_{i\in \mathrm{L} \backslash 
    (\mathrm{I}_{\upalpha}(\sigma)\cap \mathrm{L})} S^1_{(i)} \]
   where $\wedge$ and $\bigwedge$ denote the smash product with respect to the based
  spaces. So for $\mathrm{L}=\{ i_1,\cdots, i_s \}$, 
  the base-point of $\mathbf{W}^{S^1}_{\upalpha, \mathrm{L}}(\sigma)$ is
  $(e^0_{(i)},\cdots, e^0_{(i_s)})$.
   We adopt the convention that the smash product of a space with the empty space is empty.
    Then the following lemma is
    immediate from the definition of $\mathbf{W}^{S^1}_{\upalpha, \mathrm{L}}(\sigma)$.\n
   
    \begin{lem}\label{Lem:W-Simplex-Equiv} 
      For any subset $\mathrm{L} \subset [k]$
     and any simplex $\sigma\in\K$, we have
     \begin{itemize}
      \item[(i)] $\mathbf{W}^{S^1}_{\upalpha, \mathrm{L}}(\sigma) =
     \mathbf{W}^{S^1}_{\upalpha, \mathrm{L}}(\sigma\cap \K_{\upalpha,\mathrm{L}})
       $,\n
       \item[(ii)] 
     $\mathbf{W}^{S^1}_{\upalpha, \mathrm{L}}(\sigma)$ is  
     contractible whenever $\mathrm{I}_{\upalpha}(\sigma)\cap \mathrm{L}\neq \varnothing$,\n
     \item[(iii)] $\mathbf{W}^{S^1}_{\upalpha, \mathrm{L}}(\hat{\mathbf{0}}) = 
     \bigwedge_{i\in \mathrm{L}} S^1_{(i)}\cong S^{|\mathrm{L}|}$.
    \end{itemize}
    \end{lem}
   
    For any simplex $\sigma\in\K$, define
     \begin{equation}\label{Equ:D-sigma-complex}
    \mathbf{D}_{\upalpha}(\sigma) := 
      \prod_{i\in \mathrm{I}_{\upalpha}(\sigma)} S^1_{(i)}*(\sigma\cap \Delta^{\alpha_i}) 
        \times \prod_{i\in [k]\backslash \mathrm{I}_{\upalpha}(\sigma)} S^1_{(i)}.
    \end{equation}    
  Note that all the $\mathbf{D}_{\upalpha}(\sigma)$ have the same base-point
  $(e^0_{(1)},\cdots, e^0_{(k)})$, which is the base-point of
  $X(\K,\lambda_{\upalpha})$.\n

   By~\cite[Theorem 2.8]{BBCG10}, there are natural homotopy equivalences
      \[
     \mathbf{\Sigma} (\mathbf{D}_{\upalpha}(\sigma)) \simeq 
    \mathbf{\Sigma} \Big( 
      \bigvee_{\mathrm{L}\subset [k]} \mathbf{W}^{S^1}_{\upalpha, \mathrm{L}}
        (\sigma) \Big)  \]
       where $\mathbf{\Sigma}$ denotes the reduced suspension and $\bigvee$ denotes the
       wedge sum with respect to the base-point of 
       $\mathbf{W}^{S^1}_{\upalpha, \mathrm{L}}(\sigma)$. Now let
        $$
          \mathbf{E}_{\upalpha}
          (\sigma) := \bigvee_{\mathrm{L}\subset [k]} \mathbf{W}^{S^1}_{\upalpha, \mathrm{L}}
          (\sigma).$$
          
       Let $\mathrm{Cat}(\K)$ denote the \emph{face category} of $\K$ whose objects are
       simplices $\sigma\in\K$ and there is a morphism from $\sigma$ to $\tau$ whenever
       $\sigma\subseteq \tau$. Then   
     we can consider $\mathbf{D}_{\upalpha}$ and
      $\mathbf{E}_{\upalpha}$ as functors from 
      $\mathrm{Cat}(\K)$ to the category $\mathrm{CW}_*$ of connected, 
      based CW-complexes and based continuous maps. It is clear that      
      \[ X(\K,\lambda_{\upalpha}) = \bigcup_{\sigma\in \K} \mathbf{D}_{\upalpha}(\sigma) =
      \mathrm{colim}(\mathbf{D}_{\upalpha}(\sigma)).   \] 
       For any subset $\mathrm{L} =\{ l_1,\cdots, l_s\}\subset [k]$, define
         $$ \widehat{X}(\K_{\upalpha, \mathrm{L}}, \lambda_{\upalpha}) :=
         \bigcup_{\sigma\in\K} \mathbf{W}^{S^1}_{\upalpha, \mathrm{L}}(\sigma)
         =\bigcup_{\sigma\in\K_{\upalpha,\mathrm{L}}} 
         \mathbf{W}^{S^1}_{\upalpha, \mathrm{L}}(\sigma).$$
    For a fixed $\mathrm{L}\subset [k]$,
       all the spaces $\{ \mathbf{W}^{S^1}_{\upalpha, \mathrm{L}}(\sigma), \sigma\in \K_{\upalpha,\mathrm{L}}\}$ share a base-point, which then defines the base-point of 
       $\widehat{X}(\K_{\upalpha, \mathrm{L}}, \lambda_{\upalpha})$. So 
      $\widehat{X}(\K_{\upalpha, \mathrm{L}}, \lambda_{\upalpha})$ is
       the colimit of $\mathbf{W}^{S^1}_{\upalpha, \mathrm{L}}(\sigma)$ over 
       the face category $\mathrm{Cat}(\K_{\upalpha, \mathrm{L}})$ of 
       $\K_{\upalpha, \mathrm{L}}$.
       In addition, we clearly have
        $$ 
        \mathrm{colim}(\mathbf{E}_{\upalpha}(\sigma))=
           \bigvee_{\mathrm{L}\subset [k]} \widehat{X}(\K_{\upalpha, \mathrm{L}}, 
           \lambda_{\upalpha}).$$

      Since the suspension commutes with the colimits up to homotopy
      equivalence (see~\cite[Theorem 4.3]{BBCG10}), 
      we obtain the following homotopy equivalences.
      \begin{align} \label{Equ:Homm-Equv-1}
       \mathbf{\Sigma}( X(\K,\lambda_{\upalpha}) ) \simeq 
        \mathrm{colim}(\mathbf{\Sigma}(\mathbf{D}_{\upalpha}(\sigma)))
     & \simeq \mathrm{colim}(\mathbf{\Sigma} (\mathbf{E}_{\upalpha}(\sigma))) \notag \\
     & \simeq \mathbf{\Sigma} \big( \bigvee_{\mathrm{L}\subset [k]} 
           \widehat{X}(\K_{\upalpha, \mathrm{L}},\lambda_{\upalpha}) \big).
      \end{align}
      
   \begin{defi}[Order Complex]
    Given a poset (partially ordered set) $(\mathcal{P}, <)$, the \emph{order complex}
     $\Delta(\mathcal{P})$ is the simplicial complex with vertices
    given by the set of points of $\mathcal{P}$ and $k$-simplices given by the 
     ordered $(k+1)$-tuples $(p_1,p_2, . . . , p_{k+1})$ in $\mathcal{P}$ 
     with $p_1 < p_2 < \ldots < p_{k+1}$.
   \end{defi}
            
    \begin{lem} \label{Lem:Wedge-Equiv}
     For any $\mathrm{L}\subset [k]$, there is a homotopy equivalence:
      $$\widehat{X}(\K_{\upalpha, \mathrm{L}},\lambda_{\upalpha}) \simeq
      \bigvee_{\sigma\in \K_{\upalpha,\mathrm{L}}} 
      |\Delta((\K_{\upalpha, \mathrm{L}})_{<\sigma})| *
      \mathbf{W}^{S^1}_{\upalpha, \mathrm{L}}(\sigma).$$
     Here $\Delta((\K_{\upalpha, \mathrm{L}})_{<\sigma})$ is the order complex
      of the poset $\{ \tau\in \K_{\upalpha, \mathrm{L}}\,|\, \tau\supsetneq \sigma \}$ 
      whose order is the reverse inclusion,
       and $|\Delta((\K_{\upalpha, \mathrm{L}})_{<\sigma})|$ is 
      the geometric realization of $\Delta((\K_{\upalpha, \mathrm{L}})_{<\sigma})$.
    \end{lem}
    \begin{proof}
        Note that the natural inclusion
      $S^1\hookrightarrow S^1 * \Delta^{[m]}$ is
      null-homotopic for any $m\geq 1$. Then the same argument as in the proof 
      of~\cite[Theorem 2.12]{BBCG10} shows that there is a homotopy equivalence
       $H_{\mathrm{L}}(\sigma): \mathbf{W}^{S^1}_{\upalpha,\mathrm{L}}(\sigma) \rightarrow 
       \mathbf{W}^{S^1}_{\upalpha,\mathrm{L}}(\sigma)$ 
      for each simplex $\sigma\in \K_{\upalpha,\mathrm{L}}$ so that the following 
      diagram commutes for any
      $\tau\subseteq \sigma\in \K_{\upalpha,\mathrm{L}}$,
     \begin{equation}\label{Equ:Wedge-Equiv}
        \xymatrix{
          \mathbf{W}^{S^1}_{\upalpha,\mathrm{L}}(\tau) 
              \ar[r]^{H_{\mathrm{L}}(\tau)} \ar[d]_{\zeta_{\sigma,\tau}} &
              \mathbf{W}^{S^1}_{\upalpha,\mathrm{L}}(\tau)
            \ar[d]^{c_{\sigma,\tau}} \\
            \mathbf{W}^{S^1}_{\upalpha,\mathrm{L}}(\sigma)
             \ar[r]^{H_{\mathrm{L}}(\sigma)} & 
             \mathbf{W}^{S^1}_{\upalpha,\mathrm{L}}(\sigma)        
           } 
      \end{equation}  
  where $\zeta_{\sigma,\tau}$ is the natural inclusion and
  $c_{\sigma,\tau}$ is the constant map
  to the base-point. Then by~\cite[Theorem 4.1]{BBCG10} 
  and~\cite[Theorem 4.2]{BBCG10}, there is a homotopy equivalence
  $$
     \widehat{X}(\K_{\upalpha,\mathrm{L}}, \lambda_{\upalpha})
        =\mathrm{colim} (\mathbf{W}^{S^1}_{\upalpha,\mathrm{L}}(\sigma)) \ 
  \ \simeq \ 
  \bigvee_{\sigma\in \K_{\upalpha,\mathrm{L}}} |\Delta((\K_{\upalpha,\mathrm{L}})_{<\sigma})| *
      \mathbf{W}^{S^1}_{\upalpha,\mathrm{L}}(\sigma).$$
  So the lemma is proved.  \end{proof}
\nn

\begin{proof}[\textbf{Proof of Theorem~\ref{Thm:Stable-Decomp}}]
 Putting the homotopy equivalences in the equation~\eqref{Equ:Homm-Equv-1} 
 and Lemma~\ref{Lem:Wedge-Equiv} 
 together gives us a homotopy equivalence:
   \[   \mathbf{\Sigma}( X(\K,\lambda_{\upalpha}) ) \simeq 
     \mathbf{\Sigma} \Big( \bigvee_{\mathrm{L}\subset [k]} 
           \bigvee_{\sigma\in \K_{\upalpha,\mathrm{L}}}
            |\Delta((\K_{\upalpha,\mathrm{L}})_{<\sigma})| *
      \mathbf{W}^{S^1}_{\upalpha,\mathrm{L}}(\sigma) \Big). \]  
Moreover, the space in the left hand side can be simplified by the following facts.\nn
  \begin{itemize}
  \item
  $\mathbf{W}^{S^1}_{\upalpha,\mathrm{L}}(\sigma)$ is contractible whenever $\sigma\neq 
  \hat{\mathbf{0}}\in \K_{\upalpha,\mathrm{L}}$ (see Lemma~\ref{Lem:W-Simplex-Equiv}(ii)). \n
  \item
  $\Delta((\K_{\upalpha,\mathrm{L}})_{<\hat{\mathbf{0}}})$
  is isomorphic to the barycentric subdivision $\K'_{\upalpha,\mathrm{L}}$ of
   $\K_{\upalpha,\mathrm{L}}$ as simplicial complexes.
 So the geometric realization $|\Delta((\K_{\upalpha,\mathrm{L}})_{<\hat{\mathbf{0}}})|$ is homeomorphic to that of $\K_{\upalpha,\mathrm{L}}$.
  \end{itemize}
  Then by Lemma~\ref{Lem:W-Simplex-Equiv}, 
  we have the following homotopy equivalences:
    \[ \ \ \mathbf{\Sigma}( X(\K,\lambda_{\upalpha}) ) \simeq 
         \mathbf{\Sigma} \Big( \bigvee_{\mathrm{L}\subset [k]} 
            |\K_{\upalpha,\mathrm{L}}| * S^{|\mathrm{L}|}
     \Big) \simeq  \bigvee_{\mathrm{L}\subset [k]} 
           \mathbf{\Sigma}^{|\mathrm{L}|+2}( |\K_{\upalpha,\mathrm{L}}|). \] 
  So Theorem~\ref{Thm:Stable-Decomp} is proved.
 \end{proof}
                
 \vskip .4cm

    \section{A description of moment-angle complexes of simplicial posets}
     A poset (partially ordered set) $\mS$ with the order relation
    $\leq $ is called \emph{simplicial} if it has an initial element
    $\hat{0}$ and for each $\sigma \in \mS$ the lower segment
    $$[\hat{0},\sigma]=\{ \tau\in \mS: \hat{0} \leq \tau \leq \sigma\}$$
    is the face poset of a simplex.
   \n
        
    For each $\sigma \in \mS$ we assign a geometric simplex $\Delta^{\sigma}$
    whose face poset is $[\hat{0},\sigma]$, and glue these geometric simplices
    together according to the order relation in $\mS$. We get a cell
    complex $\Delta^{\mS}$ in which the closure of each cell is identified with a simplex
    preserving the face structure, and all attaching maps are inclusions.    
      We call $\Delta^{\mS}$ the \emph{geometric realization} of $\mS$.
       For convenience, we still use $\Delta^{\sigma}$ to denote the image of each
    geometric simplex $\Delta^{\sigma}$ in $\Delta^{\mS}$. Then $\Delta^{\sigma}$
    is a \emph{maximal simplex} of $\Delta^{\mS}$ 
    if and only if $\sigma$ is a maximal element of $\mS$.
  \n
      
    The notion of moment-angle complex $\mathcal{Z}_{\mS}$ associated to a 
    simplicial poset $\mS$ is introduced
    in~\cite{LuPanov11} where $\mathcal{Z}_{\mS}$ is defined via the categorical 
    language. Note that the barycentric subdivision also makes sense
    for $\Delta^{\mS}$. Let $P_{\mS}$ denote the cone of the barycentric
    subdivision of $\Delta^{\mS}$. Let the vertex set of $\Delta^{\mS}$ be
     $V(\Delta^{\mS})=\{ u_1,\cdots, u_k\}$.
     Let $\lambda_{\mS}: V(\Delta^{\mS}) \rightarrow \Z^k$ be a map so that 
     $\{ \lambda_{\mS}(u_i), \, 1\leq i \leq k\}$ is a unimodular basis of $\Z^k$.  
    Then we can construct
      $\mathcal{Z}_{\mS}$ from $P_{\mS}$ and $\lambda_{\mS}$ 
     via the same rule in~\eqref{Def:Complex-Quotient}. So we also denote
      $\mathcal{Z}_{\mS}$ by $X(\mS,\lambda_{\mS})$. \nn

      Define a map $\lambda: [m]=\{v_1,\cdots, v_m\} 
         \rightarrow \Z^m$ by $\lambda(v_i)=e_i$, $1\leq i \leq m$,
          where $\{e_1,\cdots, e_m\}$ is a unimodular basis of $\Z^m$. 
      We identify $\Delta^{[m]}$ with $\Delta^{[m]}\times \{0\}$ in 
      $\Delta^{[m]}\times[0,1]$ (considered as a product of simplices).  
      Let the vertex set of 
        $\Delta^{[m]}\times[0,1]$ be $\{v_1,\cdots, v_m,v'_1,\cdots, v'_m\}$
        where $v'_i = v_i\times \{1\}$, $1\leq i \leq m$.      
        Define a map $\tilde{\lambda}: \{v_1,\cdots, v_m,v'_1,\cdots, v'_m\} 
         \rightarrow \Z^m$ by 
         $$\tilde{\lambda}(v_i)=\tilde{\lambda}(v'_i)=e_i,\ 1\leq i \leq m.$$ 
        It is clear that $X(\Delta^{[m]}, \lambda)$ can be considered as a subspace
        $X(\Delta^{[m]}\times[0,1],\tilde{\lambda})$.
      
       \begin{lem}\label{Lem:Deform-2}
        There is a canonical deformation retraction from 
       $X(\Delta^{[m]}\times[0,1],\tilde{\lambda})$ to 
       $X(\Delta^{[m]}, \lambda)=\mathcal{Z}_{\Delta^{[m]}}$.     
        \end{lem}
      \begin{proof}
        Any
        $m$-simplex in $\Delta^{[m]}\times[0,1]$ can be written as
        $$\sigma^m_j = [v_1,\cdots, v_j,v'_j,\cdots, v'_m], \ 1\leq j \leq m.$$
        Note that $\sigma^m_j\cap \sigma^m_{j+1} = [v_1,\cdots,v_j,v'_{j+1},\cdots, v'_m]$.
        So $\sigma^m_1, \sigma^m_2 , \cdots, \sigma^m_m$ is a shelling of 
        $\Delta^{[m]}\times [0,1]$.
        By the cell decomposition~\eqref{Equ:Decomp-Quotient}, we have
        \[ X(\Delta^{[m]}, \lambda) =  
        (S^1_{(1)} *v_1) \times \cdots  \times (S^1_{(m)} *v_m)  
         \subset \prod_{1\leq j\leq m} S^1_{(j)}*[v_j,v'_j].
        \]
         \[
        X(\Delta^{[m]}\times[0,1],\tilde{\lambda}) 
          = \bigcup_{1\leq j \leq m} B_j
             \subset \prod_{1\leq j\leq m} S^1_{(j)}*[v_j,v'_j], \ \text{where} \qquad\qquad\]
        $B_j =
        (S^1_{(1)} *v_1) \times \cdots  \times (S^1_{(j-1)} *v_{j-1}) 
           \times (S^1_{(j)} *[v_j, v'_j]) \times
         (S^1_{(j+1)} *v'_{j+1}) \times \cdots \times  (S^1_{(m)} *v'_m)$.\nn
      
     There is a canonical deformation retraction from $\Delta^{[m]}\times [0,1]$
    to $\Delta^{[m]}$ along the above shelling of $\Delta^{[m]}\times [0,1]$ as follows. 
    We first compress the edge $[v_1,v'_1]$ to $v_1$, which
    induces a deformation retraction 
     $$\sigma^m_1 = [v_1,v'_1,\cdots, v'_m] \longrightarrow 
        [v_1,v'_2\cdots, v'_m] =  \sigma^m_1\cap \sigma^m_2.$$
     It compresses $\Delta^{[m]}\times [0,1] = \bigcup_{1\leq j \leq m} \sigma^m_j$ 
     to $\bigcup_{2\leq j \leq m} \sigma^m_j$.  
     Next, we compress the edge $[v_2,v'_2]$ to $v_2$ which
    induces a deformation retraction from $\sigma^m_2$ to $\sigma^m_2\cap\sigma^m_3$ and hence
    compresses $\bigcup_{2\leq j \leq m} \sigma^m_j$ to 
    $\bigcup_{3\leq j \leq m} \sigma^m_j$, and so on. After $m$ steps of retractions,  
    $\Delta^{[m]}\times [0,1]$ is deformed to $\Delta^{[m]}\times \{0\} = \Delta^{[m]}$
    (see Figure~\ref{p:3-Prism}).\n
    
     \begin{figure}
         \includegraphics[width=0.94\textwidth]{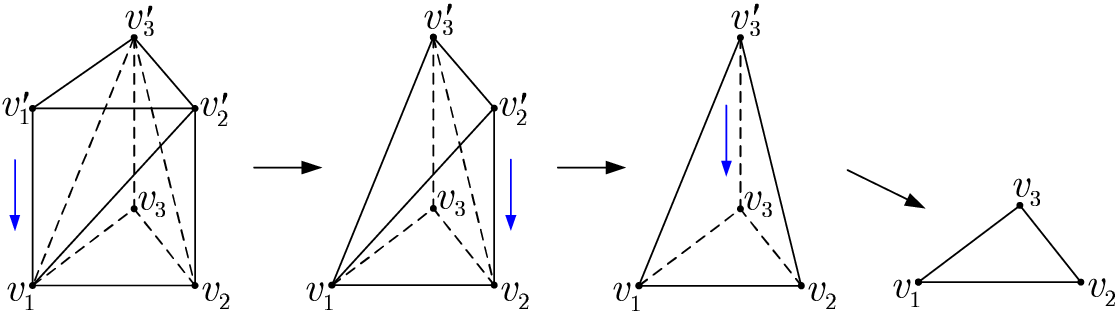}\\
          \caption{Canonical deformation retraction from $\Delta^{[m]}\times [0,1]$
    to $\Delta^{[m]}$} 
          \label{p:3-Prism}
      \end{figure}

     Using the above retractions of $\Delta^{[m]}\times [0,1]$, we obtain parallel deformation
      retractions
      from $\prod_{1\leq j\leq m} S^1_{(j)}*[v_j,v'_j]$ to $X(\Delta^{[m]}, \lambda)$
      in $m$ steps. The $j$-th step is
    to compress $S^1_{(j)} *[v_j, v'_j]$ to $S^1_{(j)} *v_j$  
       along the edge $[v_j,v_j']$, which compresses
       $B_j$ to $B_j\cap B_{j+1}$. Then starting from the first step,
     we obtain a sequence of retractions
    $$X(\Delta^{[m]}\times[0,1],\tilde{\lambda})=\bigcup_{1\leq j \leq m} B_j  \overset{\text{step 1}}{\longrightarrow} \bigcup_{2\leq j \leq m} B_j 
     \overset{\text{step 2}}{\longrightarrow} 
    \cdots \longrightarrow  \bigcup_{m-1 \leq j \leq m} B_j \overset{\text{step m-1}}{\longrightarrow}   B_m 
    $$
    $$ \overset{\text{step m}}{\longrightarrow}  
      (S^1_{(1)} *v_1) \times \cdots  \times (S^1_{(m)} *v_m) = X(\Delta^{[m]}, \lambda). $$
   The above deformation process is canonical since it only depends on the 
   ordering of the vertices of $\Delta^{[m]}$.
 \end{proof}
      
      The canonical deformation retraction from 
       $X(\Delta^{[m]}\times[0,1],\tilde{\lambda})$ to 
       $X(\Delta^{[m]}, \lambda)$ in the above lemma will serve
       as a model of homotopies in our argument below.

   \begin{thm} \label{Thm:MAC-Poset}
         For any finite simplicial poset $\mS$, there always exists a finite 
      simplicial complex
      $\K$ and a partition $\upalpha$ of $V(\K)$ so that $\mathcal{Z}_{\mS}$
      is homotopy equivalent to $X(\K,\lambda_{\upalpha})$.
      \end{thm}
      \begin{proof}
       We can construct $\K$ and $\upalpha$ in the following way. 
       Let $p: \Delta^{\mS} \times [0,n] \rightarrow \Delta^{\mS}$ be the projection
       where $\Delta^{\mS}$ is identified with $\Delta^{\mS}\times \{0\}$, and $n$ is a large
       enough integer. 
        For each maximal simplex
       $\Delta^{\sigma}\subset \Delta^{\mS}$, we 
        can choose a simplex $\widetilde{\Delta}^{\sigma}\subset \Delta^{\mS} \times 
        \{ l_{\sigma}\} $
        for some $0\leq l_{\sigma} \leq n$ so that
       \begin{itemize}
           \item $p$ maps $\widetilde{\Delta}^{\sigma}$ simplicially isomorphically onto 
                $\Delta^{\sigma}$.\n
                    
        \item $\widetilde{\Delta}^{\sigma}\cap \widetilde{\Delta}^{\tau} =\varnothing$ for any 
              maximal elements $\sigma$ and $\tau$ in $\mS$.  
        \end{itemize} 
       We call $\widetilde{\Delta}^{\sigma}$ a \emph{horizontal lifting} of $\Delta^{\sigma}$.
        We consider $\Delta^{\sigma} \times [0,n]$ 
       as the Cartesian product of $\Delta^{\sigma}$ and $[0,n]$ as 
       simplicial complexes (see~\cite[Construction 2.11]{BP02}), where $[0,n]$ is considered
       as a $1$-dimensional simplicial complex with vertices $\{0,\cdots, n\}$ and the set of
       $1$-simplices $\{ [i,i+1], 0\leq i \leq n-1 \}$.
       If $\sigma$ and $\tau$ are both maximal, 
       $\Delta^{\sigma}\cap \Delta^{\tau}$ is the geometric realization of 
       $\sigma\wedge\tau$ (the greatest common lower bound of $\sigma$ and $\tau$).
       Now define
         $$\K= \Big(\underset{\text{maximal}}{\bigcup_{\sigma\in \mS}} 
             \widetilde{\Delta}^{\sigma} \Big)
           \bigcup \Big( \underset{l_{\sigma}<l_{\tau}}{\underset{\text{maximal}}
           {\bigcup_{\sigma,\tau\in \mS}}}
            (\Delta^{\sigma} \cap \Delta^{\tau} )\times [l_{\sigma},l_{\tau}] \Big),
            $$ 
         where $[l_{\sigma},l_{\tau}]$ is considered as a simplicial subcomplex of $[0,n]$.  
     Then by our construction, $\K$ is clearly a finite simplicial complex, called a 
     \emph{stretch of $\Delta^{\mS}$} 
     (see Figure~\ref{p:Lifting} for example).          
     Let $V(\Delta^{\mS})=\{u_1,\cdots, u_k\}$ be the vertex set of $\Delta^{\mathcal{S}}$.
     Define 
      a partition $\upalpha=\{ \alpha_1,\cdots,\alpha_k\}$ of $V(\K)$
       by 
       $$\alpha_i= \{ v \in V(\K)\, |\, p(v) = u_i\}, \ 1\leq i \leq k.$$
       Then we get a $\Z^k$-coloring $\lambda_{\upalpha}$ on $\K$. In the following we show that 
       the space $X(\K,\lambda_{\upalpha})$ is homotopy equivalent to 
      $X(\mathcal{S},\lambda_{\mathcal{S}})=\mathcal{Z}_{\mathcal{S}}$.\n
      
       \begin{figure}
         \includegraphics[width=0.93\textwidth]{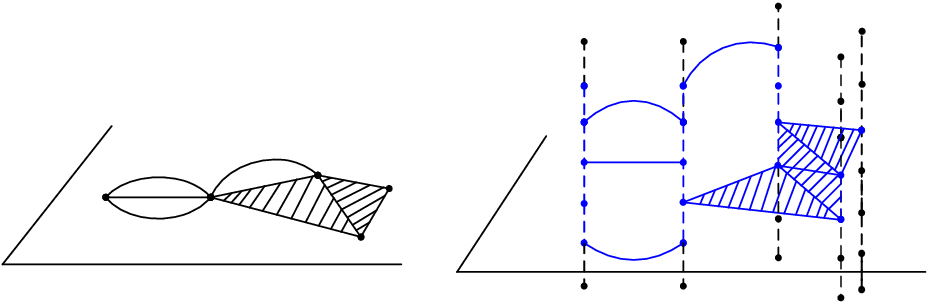}\\
          \caption{A stretch of a simplicial poset} 
          \label{p:Lifting}
      \end{figure}
      
      For a pair of maximal elements $\sigma,\tau\in \mS$, we have a decomposition
       $$(\Delta^{\sigma}\cap \Delta^{\tau}) \times [l_{\sigma},l_{\tau}] =
           \bigcup_{l_{\sigma}\leq s < l_{\tau}}
             (\Delta^{\sigma}\cap \Delta^{\tau}) \times [s,s+1]
             = \bigcup_{l_{\sigma}\leq s < l_{\tau}} \bigcup_{\omega\in \sigma\wedge\tau}
             \Delta^{\omega} \times [s,s+1].
       $$ 
      Given the shelling of each $\Delta^{\omega} \times [s,s+1]$ as we do for
       $\Delta^{[m]}\times [0,1]$
      in the proof of Lemma~\ref{Lem:Deform-2}, we obtain a canonical shelling of 
       $(\Delta^{\sigma}\cap \Delta^{\tau}) \times [l_{\sigma},l_{\tau}]$.   
      If for 
      all pairs of maximal elements $\sigma,\tau\in\mS$ we do the deformation retractions 
      from $(\Delta^{\sigma}\cap \Delta^{\tau}) \times [l_{\sigma},l_{\tau}]$ to
      $(\Delta^{\sigma}\cap \Delta^{\tau}) \times\{l_{\sigma}\}$ in $\mathcal{K}$ along their
      canonical shellings, we obtain a space that can be identified with  
       $\Delta^{\mS}$ in the end.
       Note that all these retractions are caused by 
       compressing $\{v\} \times [s, s+1]$ 
       to $\{v\}\times \{s\}$ step by step for each vertex 
       $v\in \Delta^{\sigma}\cap \Delta^{\tau}$. So for 
       different pairs of maximal elements of $\mathcal{S}$, the
       retractions we constructed are compatible with each other
       and hence can be done simultaneously at each $\{v\} \times [s, s+1]$.
       In addition, the $\Z^m$-coloring $\lambda_{\upalpha}$ on $\mathcal{K}$ 
       naturally induces a $\Z^m$-coloring on the space in each step of the
       deformation and eventually recovers 
       $\lambda_{\mathcal{S}}$ on $\Delta^{\mathcal{S}}$ in the end. So by
       applying Lemma~\ref{Lem:Deform-2} to every $\Delta^{\omega} \times [s,s+1]
       \subset (\Delta^{\sigma}\cap \Delta^{\tau}) \times [l_{\sigma},l_{\tau}]$,
       we see that the above deformation retractions on $\mathcal{K}$ induce a
       homotopy equivalence from $X(\K,\lambda_{\upalpha})$ to 
       $X(\mathcal{S}, \lambda_{\mathcal{S}})=\mathcal{Z}_{\mathcal{S}}$. 
       This proves the theorem.
     \end{proof}
      
  \n
       
       For any subset $\mathrm{L}\subset V(\mS)$, let $\mS_{\mathrm{L}}$ 
       denote the \emph{full subposet} of $\mS$
       with vertex set $\mathrm{L}$. By our construction of the stretch $\K$ of $\mS$
        in the proof of Theorem~\ref{Thm:Main-1}, the geometric realization 
        $\Delta^{\mS_{\mathrm{L}}}$ of $\mS_{\mathrm{L}}$ is a deformation retraction of
       the subcomplex $\K_{\upalpha,\mathrm{L}}$ of $\K$. So by 
      Theorem~\ref{Thm:MAC-Poset} and Theorem~\ref{Thm:Main-1},
       we derive that for any $q\geq 0$,
       $$ H^q(\mathcal{Z}_{\mS};\mathbf{k}) \cong H^q(X(\K,\lambda_{\upalpha});\mathbf{k})\cong
          \underset{\mathrm{L}\subset V(\mS)}{\bigoplus} \widetilde{H}^{q-|\mathrm{L}|-1} 
          (\K_{\upalpha, \mathrm{L}};\mathbf{k}) \cong 
          \underset{\mathrm{L}\subset V(\mS)}{\bigoplus} \widetilde{H}^{q-|\mathrm{L}|-1} 
          (\Delta^{\mS_{\mathrm{L}}};\mathbf{k}).$$
    The above equality
    can also be derived from \cite[Theorem 3.5]{LuPanov11} easily.\\

   \section{Some generalizations}
 
    We can generalize our results on $X(\K,\lambda_{\upalpha})$
   to a wider range of spaces as follows. For a
  partition $\upalpha=\{\alpha_1,\cdots, \alpha_k\}$ of $V(\K)$, 
  we can replace the $S^1_{(i)}$ in~\eqref{Equ:Decomp-Quotient}
  by a sequence of spheres $\underline{\mathbb{S}}=( S^{d_1},\cdots, S^{d_k})$ and define
  \[ \quad\ \qquad  X(\K, \upalpha,\underline{\mathbb{S}}) = \bigcup_{\sigma\in \K} \Big(
       \prod_{i\in \mathrm{I}_{\upalpha}(\sigma)} S^{d_i}*(\sigma\cap \Delta^{\alpha_i}) \times
        \prod_{i\in [k]\backslash \mathrm{I}_{\upalpha}(\sigma)} 
        S^{d_i} \Big) \subset \prod_{i\in [k]} 
        S^{d_i}*\Delta^{\alpha_i}.
  \]
  In particular when $\underline{\mathbb{S}}=( S^0,\cdots, S^0)$, the $X(\K, \upalpha,\underline{\mathbb{S}})$
  is a quotient space of the real moment-angle complex of $\K$ by the action of 
  some $\Z_2$-torus.
  We have the following two theorems which are
   parallel to Theorem~\ref{Thm:Main-1} and Theorem~\ref{Thm:Stable-Decomp}, respectively.\n
    \begin{thm} \label{Thm:Main-1-gen}
       For any coefficients $\mathbf{k}$, there is a $\mathbf{k}$-module isomorphism: 
      $$\qquad \ \ \quad H^q(X(\K, \upalpha,\underline{\mathbb{S}});\mathbf{k}) \cong \underset{\mathrm{L}
           \subset [k]}{\bigoplus} 
                   \widetilde{H}^{q-1-\sum_{i\in\mathrm{L}}d_i}
                   (\K_{\upalpha, \mathrm{L}};\mathbf{k}),\ \forall q\geq 0.$$
      \end{thm} 
      
       \begin{thm}\label{Thm:Stable-Decomp-gen}
            There is a homotopy equivalence
         $$ 
           \mathbf{\Sigma}( X(\K, \upalpha,\underline{\mathbb{S}}) )
            \simeq \bigvee_{\mathrm{L}\subset [k]} 
           \mathbf{\Sigma}^{(\sum_{i\in \mathrm{L}} d_i) +2}( \K_{\upalpha,\mathrm{L}}).$$
       \end{thm}

     The proofs of the above two theorems are completely 
     parallel to the proofs of Theorem~\ref{Thm:Main-1} and
  Theorem~\ref{Thm:Stable-Decomp}. So we leave them to the reader.
  There is only one technical point in the proof here.  
   The definition of $\kappa(i,\mathrm{L})$ and $\kappa(\sigma,\mathrm{L})$ 
   (see~\eqref{Equ:Kappa}) in the proof of Theorem~\ref{Thm:Main-1}. 
  need to be modified to be adapted 
   to Theorem~\ref{Thm:Main-1-gen}.
   For $X(\K, \upalpha,\underline{\mathbb{S}})$ we should
    redefine $\kappa(i,\mathrm{L})$  as follows and adjust
     the definition of $\kappa(\sigma,\mathrm{L})$ accordingly.
     $$\qquad\ \ \ \kappa(i,\mathrm{L})= (-1)^{r_{\underline{\mathbb{S}}}(i,\mathrm{L})}, \ \text{where}\
     r_{\underline{\mathbb{S}}}(i,\mathrm{L})= \sum_{j\in \mathrm{L},\, j< i} d_j, \ 
     \forall i\in \mathrm{L} \subset [k]. 
        $$

 \begin{rem}
 When $\upalpha^*$ is the trivial partition of
 $V(\K)$, $X(\K, \upalpha^*,\underline{\mathbb{S}})$ is nothing but the 
 polyhedral product $\K^{(\underline{\mathbb{D}},\underline{\mathbb{S}})}$ where
 $(\underline{\mathbb{D}},\underline{\mathbb{S}})= \{ (D^{d_1+1},S^{d_1}),\cdots, (D^{d_k+1},S^{d_k}) \}$.
 In this special case, Theorem~\ref{Thm:Main-1-gen} coincides 
   with~\cite[Theorem 4.2]{LuZhi09}.\\
  \end{rem}

\end{document}